\documentclass[11pt]{amsart}

\usepackage[usenames,dvipsnames]{pstricks} 
\usepackage{epsfig}
\usepackage{graphicx,color}
\usepackage{geometry}
\usepackage[all]{xy}
\usepackage{amssymb}
\usepackage{eucal}
\usepackage{cite}
\usepackage{fullpage}
\usepackage{enumerate}
\xyoption{poly}

\swapnumbers
\newtheorem{theorem}{Theorem}[section]
\newtheorem{lemma}[theorem]{Lemma}
\newtheorem{proposition}[theorem]{Proposition}
\newtheorem{corollary}[theorem]{Corollary}
\theoremstyle{definition}
\newtheorem{definition}[theorem]{Definition}

\newtheorem{remark}[theorem]{Remark}
\newtheorem{remarks}[theorem]{Remarks}
\newtheorem{example}[theorem]{Example}
\newtheorem{nonexample}[theorem]{Non-example}

\newtheorem*{remark*}{Remark}
\newtheorem*{remarks*}{Remarks}
\newtheorem*{definition*}{Definition}

\usepackage{calrsfs}

\usepackage[T1]{fontenc}
\usepackage{textcomp}
\usepackage{times}
\usepackage[scaled=0.92]{helvet}

\newcommand{\Q}{\mathbf{Q}}

\newcommand{\R}{\mathbf{R}}
\newcommand{\Z}{\mathbf{Z}}
\newcommand{\N}{\mathbf{N}}

\newcommand{\cR}{\mathcal{R}}

\newcommand{\cC}{\mathcal{C}}
\newcommand{\cP}{\mathcal{P}}
\newcommand{\cA}{\mathcal{A}}
\newcommand{\cB}{\mathcal{B}}

\newcommand{\cU}{\mathcal{U}}
\newcommand{\cJ}{\mathcal{J}}
\newcommand{\cW}{\mathcal{W}}
\newcommand{\M}{\mathcal{M}}
\newcommand{\cM}{\mathcal{M}}
\newcommand{\fX}{\mathfrak{X}}
\newcommand{\Htop}{H_{top}}
\newcommand{\htop}{h_{top}}

\newcommand{\vq}{\textbf{q}}
\newcommand{\inv}{^{-1}}
\newcommand{\nn}{\nonumber}

\usepackage{amsmath,amssymb,amsfonts,amsthm}

\begin{document}
\date{\today\ (version 0.1)}
\title[Dynamics measured in a non-Archimedean field]{Dynamics measured in a non-Archimedean field}
\author[J.\ Kool]{Janne Kool}
\address{\normalfont{Mathematisch Instituut, Universiteit Utrecht, Postbus 80.010, 3508 TA Utrecht, Nederland}}
\email{j.kool2@uu.nl}
\thanks{I like to thank Gunther Cornelissen for his support.}

\maketitle
\begin{abstract}
 \noindent
 We study dynamical systems using measures taking values in a non-Archimedean field. The underlying space for such measure is a zero-dimensional topological space. In this paper we elaborate on the natural translation of several notions, e.g., probability measures, isomorphic transformations, entropy, from classical dynamical systems to a non-Archimedean setting.
\end{abstract}

\section*{Introduction}
The study of dynamical systems using measure-theoretic methods has shown to be extremely useful. In almost all results real or complex valued measures are used. However, there exists theories in which measures take values in other fields, in particular non-Archimedean fields. In a series of papers \cite{Monna} Monna started the study of non-Archimedean functional analysis. Integration theory using measures for non-Archimedean valued functions on locally compact topological zero dimensional spaces was developed by Monna and Springer \cite{MonnaSpringerI} \cite{MonnaSpringerII}, and later generalized to all zero-dimensional topological spaces by Van Rooij and Schikhof \cite{VanRooijSchikhof}. A nice overview of non-Archimedean functional analysis, including measure theory, can be found in Van Rooij \cite{VanRooij}. 

Recently, non-Archimedean analysis, and in particular measure theory, found several applications to theoretical physics \cite{kyr},\cite{Khrennikov}. In this paper we elaborate on the natural translation of several notions, e.g., probability measures, isomorphic transformations, entropy, from classical dynamical systems to a non-Archimedean setting.

Let us explain why it is expected that these notions behave differently. For a discrete complete non-Archimedean field $K$, we cannot expect that a $K$-valued Borel measure is $\sigma$-additive. On the contrary, any $\sigma$-additive $K$-valued function on the Borel algebra is trivial, i.e., such a function is a, possibly infinite, sum of Dirac measures (\!\cite{VanRooij}, lemma 4.19). To overcome this problem, instead of $\sigma$-algebras, separating covering rings are used. These rings form a basis for a zero-dimensional Hausdorff topological  space.

A measure $\mu:\cR\to K$ is then an additive map satisfying some boundedness and continuity condition (see \ref{measure} for the exact definition). It comes with a real valued function,
\[\|\cdot\|_\mu:\cR\to\R, A\mapsto\sup\{|\mu(B)|:B\in\cR,B\subset A\}.\]      
Sets $A\in\cR$ are \emph{negligible} if $\|A\|_\mu=0$. One of the most eye-catching distinctions between classical measures and these measures is that there exists a set $X_0(\mu)$, which is the \emph{biggest} negligible set, i.e., $A$ is negligible if and only if $A\subset X_0(\mu)$. The real-valued function $\|\cdot\|_\mu$ induces a seminorm on the space of $K$-valued functions on $X$:
\[\|f\|_\mu=\sup\{|f(x)|N_\mu(x):x\in X\}\]
where $N_\mu(x)=\inf\{\|A\|_\mu: A\in\cR, x\in A\}$. This seminorm is used to find  the integrable functions with respect to $\mu$. Let $S(X)$ be the $K$-vector space of step functions, i.e., the space of finite linear combinations of characteristic functions $\chi_A$ for $A\in\cR$. Integration with respect to $\mu$ is defined as the unique functional such that for every $A\in\cR$, $\int_X\chi_Ad\mu=\mu(A)$.  A function $f:X\to K$ is integrable if there exists a sequence $\{f_n\}$ of step functions in $S(X)$ such that $\lim_{n\to\infty}\|f_n-f\|_\mu=0$. This process leads to an extension of $\cR$ to a covering ring $\cR_\mu$ which contains all sets for which the characteristic function is integrable. The set is of integrable functions restricted to $X^+=X\backslash X_0(\mu)$ is a Banach space for the induced norm $\|\cdot\|_\mu$, denoted by $L^1(\mu)$. A triple $(X,\cR_\mu,\mu)$ is called a \emph{probability space} if $\cR_\mu$ is an algebra, and if $\mu(X)=1$. 

The aim of this paper is to develop the theory of dynamical systems on these probability spaces. We call a measurable map $T:X\to X$ \emph{measure preserving} if for any $A\in\cR$, $\mu(T\inv A)=\mu(A)$, and we call a four-tuple $(X,\cR_\mu,\mu,T)$ a \emph{dynamical system}. The first noticeable property is that the biggest negligible set $X_0(\mu)$ is invariant under any measure preserving transformation. This is very different from the classical situation: we can totally neglect $X_0(\mu)$ by restricting to $(X^+,\cR_\mu^+,\mu,T)$, where $\cR_\mu^+$ is the ring $\{A\cap X^+:A\in\cR_\mu\}$. A consequence is that the direct analog of the notion of ergodicity in the classical sense is not very useful; it reduces in this setting to the statement that any $T$-invariant subset of $X$ is negligible or contains the full set $X^+$.  

The notion of isomorphic dynamical systems, however, is still useful. In fact, neglecting $X_0(\mu)$ is also what happens in our definition; two dynamical systems $(X,\cR_\mu,\mu,T)$ and $(Y,\cR_\nu,\nu,S)$ are called \emph{isomorphic} if there is a measure preserving $(\cR_\mu,\cR_\nu)$-homeomorphism $\phi:X^+\to Y^+$ such that $\phi\circ T=S\circ\phi$. Completely analogous to the classical setting we call these two dynamical systems \emph{conjugate} if the is a measure algebra isomorphism $\Phi:(\cR_\mu,\mu)\to(\cR_\nu,\nu)$ such that $\Phi\inv\circ S\inv=T\inv\circ\Phi\inv$. In the classical theory isomorphy implies conjugacy, but not conversely. In our theory we find that isomorphy and conjugacy are in fact equivalent (see theorem \ref{isomisconj}). 

A measure preserving transformation induces a linear map $U_T:L^1(\mu)\to L^1(\mu), f\to f\circ T$ which is an isometry if $T$ is invertible. The definition of \emph{spectral isomorphy} of two invertible measure preserving transformations $T:X\to X$ and $S:Y\to Y$ comprises an isometry $W:L^1(\mu)\to L^1(\nu)$ such that $U_S\circ W=W\circ U_T$. If two dynamical systems are isomorphic, then they are also spectral isomorphic. We give some conditions under which spectral isomorphy implies conjugacy, and hence isomorphy (see lemma \ref{conditions}). It is an interesting problem if these are necessary. 

In the last section we develop a notion of non-Archimedean \emph{measure entropy}, which is an invariant under isomorphisms. Let $\alpha$ be a partition of $X$ by elements of $\cR_\mu$ and let $M(\alpha)$ be the number of elements of $\cM(\alpha)=\{A\in\alpha:\|A\|>0 \}$. The measure entropy is the defined by $H_\mu(\alpha)=\min\{\|A\|\log M(\alpha): A\in\cM(\alpha)\}$.  This measure entropy is connected to the topological entropy, for the topology induced by $\cR_\mu$. For compact $X^+$ and measures $\mu$ such that for all nonempty $A\in\cR_\mu^+$, $|\mu(A)|=1$ if the measure entropy equals the topological entropy (see theorem \ref{entropy}). 

Let us also digress on some notions from dynamical systems of which we don't know how to translate them to the non-Archimedean setting; Poincar\'e recurrence and the Birkhoff ergodic theorem. Poincar\'e recurrence states that in a probability space for any non-negligible measurable set $A$ the subset of $A$ of elements which are not recurrent under a measure preserving $T$ is negligible. The classical proof relies heavily on the measurability of the set of non-recurrent points. The non-Archimedean measures, however, are in general not $\sigma$-additive. In fact, it is possible to construct examples for which the recurrent set is not measurable and indeed not negligible.

The Birkhoff ergodic theorem assures the convergence in $L^1$ of the average 
\[\frac{1}{n}\sum_{i=0}^{n-1}f(T^i(x)),\]
to a $T$-invariant function in $L^1$. However, in a non-Archimedean field there is not a notion of average. In particular, the sequence $\{\frac{1}{n}\}_{n\in\N}$ is not convergent. 

Throughout we discuss several examples.

\section{Non-Archimedean measures and integration theory}
\subsection*{Measures}
We start by explaining the main set up, as described in (\!\cite{Van Rooij}, chapter 7). Instead of $\sigma$-algebras, separating covering rings are used. For any set $X$, denote by $\cP(X)$ its power set.

\begin{definition}\label{measure}
A collection $\cR\subseteq\mathcal P(X)$ is called a \emph{covering ring} if it has the following properties: 
\begin{enumerate}
\item if $A,B\in \cR$ then $A\cap B$, $A\cup B$ and $A\backslash B$ are in $\cR$. 
\item $\cR$ covers $X$.
\end{enumerate}
Such a ring $\cR$ is called \emph{separating} if for any distinct $x,y\in X$, there is an $A\in\cR$ such that $x\in A,$ and $y\not\in A$. A covering ring is an \emph{algebra} if $X\in\cR$.
\end{definition}	

A covering ring $\cR $ is the basis of a zero dimensional topology--in which the elements of $\cR$ are closed and open. This so called $\cR$-topology is Hausdorff if and only if $\cR$ is separating. In the text below all covering rings are separating. A subcollection $\mathcal A\subset \cR$ is called \emph{shrinking} if the intersection of any two elements of $\mathcal A$ contains an element of $\mathcal A$. 

\begin{definition}
A map $\mu:\cR\to K$ is called a \emph{measure}, if
\begin{description}
\item[\emph{additive}] for disjoint $A,B\in\cR$, $\mu(A\cup B)=\mu(A)+\mu(B)$.
\item[\emph{bounded}] for all $A\in \cR$ the set $\{\mu(B): B\in \cR, B\subset A\}$ is bounded,
\item[\emph{continuous}] if $\mathcal A$ is shrinking and $\bigcap\limits_{A\in\cA}A=\emptyset$, then $\lim\limits_{A\in\mathcal A}\mu(A)=0$ 
\end{description} 
The latter limit is defined as follows: for every $\epsilon>0$ there is a $A_0\in\mathcal A$, such that for every $A\in\mathcal A$ contained in $A_0$, $|\mu(A)|<\epsilon$.
\end{definition}
The continuity property of the measure is the replacement for $\sigma$-additivity.  

\begin{lemma}
Let $\{B_n\}_{n\in\N} \subset\cR$ a collection of disjoint sets such $\bigcup_{n\in\N}B_n\in\cR$, then $\mu(\bigcup_{n\in\N}B_n)=\sum_n\mu(B_n)$. 
\end{lemma}

\begin{proof}
 Define $A_n=\bigcup_{i=n}^\infty B_i$, then $\{A_i\}_{i\in\N}$ forms a shrinking collection with empty intersection. Moreover, $B_i=A_i\backslash A_{i+1}$ and $\mu(B_i)=\mu(A_i)-\mu(A_{i+1})$. Hence 
\[\mu(\bigcup_iB_i)=\lim_{n\to\infty}\sum_{i=1}^n\mu(B_i)=\lim_{n\to\infty}\sum_{i+1}^n\mu(A_i)-\mu(A_{i+1})=\mu(A_1).\]\end{proof}
Let us illustrate these measures with some examples, and also a non-example to illustrate the necessity of the continuity condition.

\begin{example}\label{Qp}
In our first example $X=\Q_p$, the field of $p$-adic numbers. It is a valued field, and the valuation $v_p$ induces a metric $|x-y|_p=p^{-v_p(x-y)}$. This metric is non-Archimedean, and, therefore, balls of the form $B_{p^n}(x)=\{y\in X:|x-y|<p^n\}$ are both open and closed. The collection $B_c(X)$ of all compact clopen subsets of $X$ is a covering ring. Let $q$ be a prime number. There is an unique additive function $\mu:B_c(X)\to\Q_q$ such that $\mu(B_{p^n}(x))=p^n$. If $q\not=p$, then $\mu$ is a measure, which takes values in $\Z_q^*\cup\{0\}$, where $\Z_q^*$ denotes the ring of invertible $q$-adic integers.  If $p=q$, then $\mu$ is not a measure, because the boundedness condition for measures is violated. For instance, the values of $\mu$ of the sequence of sets 
\[B_{p^0}(0)\supset B_{p^{-1}}(0)\supset....\supset B_{p^{-n}}(0)\supset...,\] 
form a sequence with increasing, and even unbounded, absolute values.
\end{example}

\begin{example}\label{h}
Let $X$ be any set, and let $\cR$ be the ring which consists of all finite subsets of $X$. Let $h:X\to K$ be any function. Then $\kappa:\cR\to K, A\mapsto\sum_{a\in A}h(a)$ is a measure. More generally, let $h:X\to K$ be any function, and let $\cR_h$ be the ring which consists of all subsets $A$ of $X$ for which $\sum_{a\in A}h(a)$ converges. Define the measure $\kappa$ as above.
\end{example}

\begin{nonexample}\label{nonexample}
This example is a non-example in the following sense. The map described below is additive and bounded, but we will later prove that it lacks the continuity condition. Let $p$ be a prime and let  $\fX_p^k=\{\frac{n}{p^k}:0\leq n\leq p^k\}$ and $\fX_p=\cup_{k\geq0}\fX_p^k$. Finally define $X_p=[0,1]\backslash\fX_p$. For $r,s\in\fX_p, r<s$ let $I_{r,s}$ be the open interval $(r,s)\subset [0,1]$. Then the collection $J_{r,s}=I_{r,s}\cap X_p$ generates a covering algebra $\cJ$ of $X_p$. Define 
\[\upsilon:\cJ\to\Q_p, \upsilon(J_{r,s})=\left\{\begin{array}{cc}\frac{1}{s}-\frac{1}{r}& \text{ if } r\not=0\\ \frac{1}{s}& \text{ if } r=0\end{array}\right.\]
The map $\upsilon$ is bounded because $v_p(\frac{1}{x})\geq 0$ for all $x\in\fX_p$, and hence $v_p(\upsilon(J_{r,s}))\geq 0$ for all $r,s\in\fX_p$. 
\end{nonexample}

One of the main losses of using $K$-valued measures is that $K$ is not ordered, and therefore, the measure doesn't order sets into "bigger" or "smaller". Moreover, it could very well happen that there are sets of measure zero which contain sets of non-zero measure. To overcome these problems, a measure comes with a real valued function:
\[\|\cdot\|_\mu:\cR\to\R:A\mapsto\sup\{|\mu(B)|: B\in\cR, B\subset A\}.\]
In case confusion is unlikely we will suppress the subscript $\mu$. This $\R$-valued function is not at all like a real valued measure in the classical sense; it possesses the following properties:
\begin{description}
 \item[\emph{monotone}] if $A\subset B$, then $\|A\|\leq\|B\|$,
 \item[\emph{convex}] for any $A,B\in\cR$, $\|A\cup B \|\leq\max(\|A\|,\|B\|)$.
 \item[\emph{minimum}] for $\|A\cap B\|\leq\min(\|A\|,\|B\|)$
\end{description}

\begin{lemma}[\!\cite{VanRooij}, page 249]\label{equivalentcontinuity}
The continuity property of a measure is equivalent to the following assertion. If $\cA\subset\cR$ is shrinking and $\bigcap\limits_{A\in\cA} A=\emptyset$, then $\lim\limits_{A\in\cA}\|A\|=0$.  
\end{lemma}

Define the \emph{norm function} $N_\mu:X\to[0,\infty)$ by
\[N_\mu(x)=\inf\{\|U\|_\mu:U\in\cR, x\in U\}.\]
The reason that this function is called the norm function is because it is used to define a seminorm on the space of $K$-valued functions on $X$. For $f:X\to K$ define
\[\|f\|_\mu=\sup_{x\in X}|f(x)|N_\mu(x).\]
This apparent abuse of notation is justified by the following lemma.

\begin{lemma}[\!\cite{VanRooij}, Lemma 7.2]\label{normindicator}
For the indicator function $\chi_B$ of any $B\in\cR$, $\|\chi_B\|_\mu=\|B\|_\mu$. 
\end{lemma}

We denote the level set $\{x\in X:N_\mu(x)=0\}$ by $X_0(\mu)$.
\begin{definition}
Subsets of $X_0(\mu)$ are called \emph{negligible}.
\end{definition}

\begin{lemma}\label{equivalentverwaarloosbaar}
For any $A\in\cR$ the following properties are equivalent.
\begin{enumerate}
 \item\label{negligible} $A$ is negligible,
 \item\label{nul} $\|A\|=0$,
 \item\label{intersection} for all $B\in\cR$, $\mu(A\cap B)=\mu(A)$.
\end{enumerate}
\end{lemma}

\begin{proof}
We first prove that (\ref{negligible})$\Leftrightarrow$(\ref{nul}). 
Let $A$ be a negligible set. Then by lemma (\ref{normindicator}) $\|A\|=\|\chi_A\|=\sup_{x\in A}N_\mu(x)$, hence $\|A\|=0$ if and only if $A$ is negligible.

Next we prove (\ref{nul})$\Leftrightarrow$(\ref{intersection}). Suppose that $\|A\|=0$, then $\mu(A)=0$ and for all $B\in\cR$, $A\cap B\subset A$, and hence $\mu(A\cap B)=0$. Conversely, for any $B\in\cR$ with $B\subset A$, $B=B\cap A$, and since $\emptyset\in\cR$, $\mu(B)=\mu(A)=\mu(\emptyset )=0$. It follows that $\|A\|=0$.
\end{proof}

\begin{example}
Let us have a closer look to the examples discussed above. In example (\ref{Qp}) we have $X_0(\mu)=\emptyset$, while in $(\ref{h})$, $X_0(\kappa)=\{x\in X:h(x)=0\}$; in fact, $N_\kappa(x)=|h(x)|$. 
\end{example}

\begin{nonexample}
We determine $N_\upsilon(x)$ on $X_p$ in non-example (\ref{nonexample}). Any element $x\in [0,1]$ can be represented by
\[x=\sum_{i=1}^\infty\frac{a_i}{p^i}, \text{ where } a_i\in\{0,...,p-1\}.\]
 It is well known that such representation are sometimes not unique. Elements with representations with coordinates that are constant eventually, are in $\fX_p$. Let $x\in X_p$ with a expansion $x=\sum_{i=1}^\infty a_i/p^i$. Then for any $n\in\N$, 
\[\sum_{i=1}^n\frac{a_i}{p^i}<x<\frac{1}{p^n}+\sum_{i=1}^n\frac{a_i}{p^i}.\] 
We call this interval $J_n(x)$, and calculate $\upsilon(J_n(x))$ for a $n$ for which $a_n\not=0$. 
Denote $r=\sum_{i=1}^n\frac{a_i}{p^i},t=\frac{1}{p^n}$, and let $v_p$ be the p-adic valuation. As $v_p(r)=v_p(t)$ we find
\[v_p(\upsilon(J_n))=v_p(\frac{1}{r+t}-\frac{1}{t})=v_p(r)-(v_p(t)+v_p(r+t))=-v_p(r+t).\]
Define $k_n=-v_p(r+t)=\max\{k:a_k\not=p-1, k\leq n\}$. Since the coordinates of $x$ can not be constant $p-1$ eventually, it follows that 
\[\lim_{n\to\infty}\upsilon(J_n(x))=\lim_{n\to\infty}p^{k_n}=0.\]
In particular, this shows that $N_\upsilon(x)=0$ for all $x\in X_p$, i.e., the entire set $X_p$ is negligible. This finishes the proof that $\upsilon$ does not satisfy the continuity condition, because, for instance, 
\[\|\chi_{X_p}\|_{\upsilon}=\sup_{x\in X_p}N_\upsilon(x)=0\not=\|X_p\|_\upsilon,\]
which contradicts lemma (\ref{normindicator}). 
\end{nonexample}

\subsection*{Integration}
Analogous to the classical integration theory, integrals with respect to a measure are defined by approximation by step functions.

\begin{definition}
A $\cR$-\emph{step-function} is a finite $K$-linear combination of indicator functions $\chi_A$ of elements in $A\in\cR$. 
\end{definition}
 
Note that step functions can be written as a finite linear combination of indicator functions of disjoint elements of $\cR$. The step functions $S(X)$ form a $K$-vector space. The integral is the unique linear functional $S(X)\to K$ for which
\[\int_X\chi_A(x)d\mu(x)=\mu(A).\] 
for any $A\in\cR$. It satisfies the inequality
\begin{align}\label{integraalnorm}|\int_Xf(x)d\mu(x)|\leq\|f\|_{\mu}.\end{align}
A function $f:X\to K$ is called $\mu$-\emph{integrable} if there exists a sequence $\{f_n\}_{n\in N}$ of step functions such that $\lim_{n\to\infty}\|f_n-f\|_\mu\to0$. The space of integrable functions is a vector space, denoted by $L(\mu)$. The integration functional is extended to $L(\mu)$ by continuity. Inequality (\ref{integraalnorm}) holds for this extension.

A set $A\subset X$ is \emph{measurable} if its indicator function $\chi_A$ is in $L(\mu)$. The collection $\cR_\mu$ of measurable sets is characterized by the following lemma.

\begin{lemma}[\!\cite{VanRooij}, Lemma 7.3]
A set $A\in\cP(X)$ is an element of $\cR_\mu$ if and only if for every $\epsilon>0$ there exists a $B\in\cR$ such that on the symmetric difference $A\Delta B$, $N_\mu\leq\epsilon$. 
\end{lemma}

It follows that $\cR_\mu$ itself again forms a covering ring. Clearly, $\cR$ is contained in $\cR_\mu$, and the $\cR_\mu$-topology is finer than the $\cR$-topology. In particular all subsets of $X_0(\mu)$ are measurable. The ring $\cR_\mu$ is maximal in the following sense: repetition of the procedure above would lead again to $\cR_\mu$, i.e, $(\cR_\mu)_\mu=\cR_\mu$. 
\begin{definition}
A triple $(X,\cR_\mu,\mu)$ is called a \emph{measure space}.	 
\end{definition}

Two functions $f,g$ are said to be equal $\mu$-almost everywhere, if $f(x)=g(x)$ for all $x\in X$ except maybe on a subset of $X_0$. Being equal $\mu$-almost everywhere defines an equivalence relation on $L(\mu)$ denoted by $\sim_\mu$. 

\begin{definition}The space $L^1(\mu)$ to be $L(\mu)$ modulo the relation $\sim_\mu$, equipped with the norm induced by the seminorm $\|\cdot\|_\mu$ on $L(\mu)$.
\end{definition}
 
\section{Measure preserving transformations}
We study the dynamics on measure spaces, where $\mu$ is a probability measure. 

\begin{definition}
A measure $\mu:\cR\to K$ is called a \emph{probability measure} if
\begin{enumerate}
 \item  the covering ring $\cR$ is an algebra, i.e., it is a covering ring such that  $X\in\cR$.
 \item $\mu(X)=1$.
\end{enumerate}
A measure space $(X,\cR_\mu,\mu)$ is called a \emph{probability space} if $\mu$ is a probability measure. 
\end{definition}

\begin{definition} 
Let $T:(X,\cR_\mu,\mu)\to(Y,\cR_\nu,\nu)$ be a map. It is called \emph{measurable} if $T$ is continuous relative to the $\cR_\mu$ and $\cR_\nu$-topologies. It is called \emph{measure preserving} if it is measurable and if for any $B\in\cR_\nu$, $\mu(T\inv B)=\nu(B)$. A measure preserving $T:(X,\cR_\mu,\mu)\to(X,\cR_\mu,\mu)$ is called a \emph{measure preserving transformation}. A measure preserving transformation is \emph{invertible} if it is a homeomorphism, and if for all $B\in\cR_\mu$, $\mu(TB)=\mu(B)$.
\end{definition}

\begin{definition}
A map $\Phi:(\cR_\mu,\mu)\to(\cR_\nu,\nu)$ between two measure algebras is called a \emph{measure algebra isomorphism} if $\Phi$ is a bijection which preserves complements and unions and $\nu(\Phi(B))=\mu(B)$ for all $B\in\cR_\mu$.
\end{definition}

Clearly, a measure algebra isomorphism also preserves inclusions and intersections. 

\begin{lemma}\label{invalgisom}
Let $\Phi:(\cR_\mu,\mu)\to(\cR_\mu,\mu)$ be a measure algebra isomorphism. The map  $\|\cdot\|_\mu$ is invariant under $\Phi$.
\end{lemma}

\begin{proof}
Let $\Phi$ be a measure algebra isomorphism and let $F\in\cR_\mu$ then
\begin{align}
 \|\Phi\inv F\|&=\sup\{|\mu(B)|:B\in\cR_\mu,B\subset \Phi\inv F\}\nn\\
 			&=\sup\{|\mu(\Phi\inv B')|:B'\in\cR_\mu,B'\subset F\}\nn\\
			&=\sup\{|\mu(B')|:B'\in\cR_\mu,B'\subset F\}=\|F\|.\nn
\end{align}
\end{proof}

\begin{corollary}
The biggest negligible set $X_0(\mu)$, is invariant under measure preserving transformations. 
\end{corollary}

\begin{proof}
Recall from lemma \ref{equivalentverwaarloosbaar} that $X_0(\mu)$ is the maximal set $A$ for which $\|A\|_\mu=0$.
\end{proof}

\begin{lemma}
The norm map $N_\mu$ is invariant under any \emph{invertible} measure preserving transformation $T:(X,\cR_\mu,\mu)\to(X,\cR_\mu,\mu)$.  
\end{lemma}

\begin{proof}
 \begin{align}
 N_\mu(x)&=\inf_{U\in\cR_\mu, x\in U}\|U\|\nn\\
		&=\inf_{TU\in\cR_\mu, x\in U}\|U\|  \text{ (rename } TU=W\text{)}\nn\\
		&=\inf_{W\in\cR_\mu, x\in T\inv W}\|T\inv W\|\nn\\
		&=\inf_{W\in\cR_\mu, x\in T\inv W}\|W\|\nn\\
		&=\inf_{W\in\cR_\mu, Tx\in W}\|W\|=N_\mu(Tx).\nn
 \end{align}
\end{proof}

Let us give an example of a (non-invertible) measure preserving map, for which $N_\mu(x)$ is not invariant.

\begin{example}\label{shift}Let $\Omega=\{0,...,p-1\}^\N$ be the space of one sided infinite words $w=w_0w_1...$ in the alphabet $\{0,...,p-1\}$. The map 
\[\Z_p\to\Omega:\sum_{i\geq 0}a_ip^i\mapsto w=a_0a_1a_2...\]
is a bijection. The restriction to $\Z_p$ of the measure $\mu$ with values in $\Q_q$ from the first example (\ref{Qp}) induces a measure $\mu_q$ on $\Omega$. Let us describe this measure. Let $\omega=a_0...a_{n-1}$ be a finite word in the same alphabet. Define 
\[U_{\omega}=\{w\in\Omega:w_0=a_0,...,w_{n-1}=a_{n-1}\}.\]
The set $U_{\omega}$ is called a \emph{cylindrical set}, and it is the image of $\sum_{i=0}^{n-1}a_ip^i+p^n\Z_p$ of the bijection above.  We find, $\mu_q(U_\omega)=p^n$. Of course, $\Omega$ is very similar to the classical one sided shift space. In fact the shift map
\[\sigma:\Omega\to\Omega, w_0w_1w_2...\mapsto w_1w_2w_3...,\]
is measure preserving. Clearly, $\sigma$ is not injective, and therefore, not invertible. Again, very similar to the classical Bernoulli-shift, more shift invariant measures can be found. Let $\textbf{q}=(q_0,...,q_{p-1})$ be a vector in $\Q_q^p$, such that $\sum_{i=0}^{p-1}q_i=1$, and for all $i$, $|q_i|_q\leq1$. Let $\mu_{\textbf{q}}$ be the measure on $\Omega$ such that 
\[\mu_{\textbf{q}}(U_\omega)=q_{a_0}...q_{a_n}.\]
This measure is invariant under the shift, for 
\[\sigma\inv(U_{\omega})=\bigcup_{i=0}^{p-1}U_{i\omega},\]
and
\[\mu_{\textbf{q}}(\bigcup_{i=0}^{p-1}U_{i\omega})=\sum_{i=o}^{p-1}q_i\mu_{\textbf{q}}(U_\omega)=\mu_{\textbf{q}}(U_\omega).\]
Now choose $p=2, q=3$, and let $\textbf{q}=(-2,3)$. Let $U_\omega$ be a cylindrical set in $\Omega$, as for any  cylindrical set $U_{\omega'}\subset U_\omega$, $|\mu_{\textbf{q}}(U_{\omega'})|\leq|\mu_{\textbf{q}}(U_{\omega'})|$,  $\|U_\omega\|=|\mu_{\textbf{q}}(U_\omega)|=3^{-\#\{i:\omega_i=1\}}$. The biggest negligible set is given by  $\Omega_0(\mu_\vq)=\{w\in\Omega:\#\{i:w_i=1\}=\infty\}$. For $w\not\in\Omega_0(\mu_\vq)$, $N_{\mu_\vq}(w)=3^{-\#\{i:w_i=1\}}$. The function $N_{\mu_\vq}$ is not invariant under $\sigma$; if $w$ is such that $w_0=1$ then $N_{\mu_\vq}(\sigma w)=3N_{\mu_\vq}(w)$.
\end{example}

\section{Isomorphisms and spectral isomorphisms}\label{isomorphisms} In classical dynamics, several equivalence relations on the collection of dynamical systems are distinguished, for instance, isomorphy, conjugacy and spectral isomorphy. We discuss the analogies of these notions, and, remarkably, we will show that isomorphy and conjugacy turn out to be the same.

\begin{definition}
Let $(X_1,\mu,\cR_\mu)$ and $(X_2,\nu,\cR_\nu)$ be probability spaces, and let there be measure preserving transformations $T_i:X_i\to X_i$, $i=1,2$. Then $T_1$ and $T_2$ are called \emph{isomorphic}, $T_1\cong T_2$, if there are sets $M_i\subset X_i$, such that $X_i\backslash M_i$ is negligible  and $T_iM_i= M_i$ for $i=1,2$, and if there is a measure preserving transformation $\phi:M_1\to M_2$ such that $\phi\circ T_1(x)=T_2\circ\phi(x)$ for all $x\in M_1$. Equivalently, $T_1$ and $T_2$ are isomorphic if there is a measure preserving transformation $\phi: X_1^+\to X_2^+$ such that $\phi\circ T_1=T_2\circ \phi$. 
\end{definition}

\begin{definition}
Two measure preserving transformations $T_1:(X_1,\cR_\mu,\mu)\to(X_1,\cR_\mu,\mu)$ and $T_2:(X_2,\cR_\nu,\nu)\to(X_2,\cR_\nu,\nu)$ are called \emph{conjugate}, $T_1\sim T_2$, if there is a measure algebra isomorphism $\Phi:(\cR_\mu,\mu)\to(\cR_\nu,\nu)$ such that $\Phi\inv\circ T_2\inv=T_1\inv\circ\Phi\inv$.
\end{definition}

\begin{theorem}\label{isomisconj}
Two measure preserving transformations are isomorphic if and only if they are conjugate. 
\end{theorem}

\begin{proof}
An isomorphism $\phi$ induces an measure algebra isomorphism $\Phi$ directly; define for $A\in\cR_\mu$, $\Phi(A)=\phi(A)$.

Conversely, since for any $A\in\cR_\mu$, $\|A\|_\mu=\|\Phi(A)\|_\nu$, we may restrict ourself to $X^+$. Choose $x\in X^+$. Since $X^+$ equipped with the $\cR_\mu$-topology is a Hausdorff space, the singleton $\{x\}$ is closed. In particular, $X\backslash\{x\}$ is open, and therefore, there is a sequence $\{A_i\}_{i\in\N}\subset\cR_\mu$ such that $X\backslash\{x\}=\cup_{i}A_i$. Then \[B_n=X\backslash\bigcup_{i=1}^nA_i\]
defines a descending sequence in $\cR_\mu$, such that $\bigcap_nB_n=\{x\}$. Because $\Phi$ preserves intersections $\{\Phi(B_n)\}_{n\in\N}$ is a descending sequence as well. We would like to define 
\[\phi(x)=\bigcap_{n}\Phi(B_n),\]
however, we should check that this intersection is a singleton. To do so, first suppose it is empty, then $\cB=\{\Phi(B_n)\}_{n\in\N}$ is a shrinking collection, and by the continuity condition in lemma \ref{equivalentcontinuity} on $\nu$ $\lim_{\cB}\|B\|_\nu=0$, i.e.,  there is for any $\epsilon>0$ a $N\in\N$ such that for any $n>N$,
\[\|(\Phi(B_n))\|_\nu=\|B_n\|_\mu<\epsilon.\]
So, in particular, also for $0<\epsilon<N_\mu(x)$. But  this contradicts that for all $n\in\N$ we have $N_\mu(x)\leq\|B_n\|_\mu$. 
Secondly, we check that $\cap_n\Phi(B_n)$ contains at most one element. Suppose $y\in\cap_n\Phi(B_n)$, let $U_y\in\cR_{\nu}$ contain $y$. Then $\Phi\inv(U_y)\cap B_n\not=\emptyset$ for all $n\in\N$. Hence $x\in\Phi\inv(U_y)$. Now suppose that $y'\in\cap_n\Phi(B_n)$ and $y\not=y'$. Because $\cR_\nu$ is a separable ring, it is possible to choose $U_y$, $U_{y'}$ containing $y$ and $y'$ respectively, such that $U_y\cap U_{y'}=\emptyset$, and hence $\Phi\inv(U_y)\cap\Phi\inv(U_{y'})=\emptyset$. However, $x$ is contained by both $\Phi\inv(U_y)$ and $\Phi\inv(U_{y'})$. It follows that $\phi$ is a well defined map. 

It is left to check that $\phi$ is indeed an isomorphism. Let $x$ and $B_n$ be as above, then \[T_1\inv(x)=T_1\inv(\bigcap_nB_n)=\bigcap_n(T_1\inv B_n),\]
and hence,
 \[\phi(T_1\inv(x))=\bigcap_n(\Phi\circ T_1\inv B_n)=\bigcap_n(T_2\inv\circ\Phi B_n)=T_2\inv(\bigcap_n\Phi B_n)=T_2\inv(\phi(x)).\]
Therefore, $\phi(T_1(x))=T_2(\phi(x))$, for all $x\in X^+$. That $\phi$ is measure preserving follows since for any $A\in\cR_\mu$, $\phi(A)=\Phi(A)$. 
\end{proof}

A measure preserving transformation $T:X\to X$ induces an operator $U_T:L^1(\mu)\to L^1(\mu),f\mapsto f\circ T$. Classically, $U_T$ is unitary if and only if $T$ is invertible. The analogous statement here is weaker. 

\begin{lemma}
If $T$ is an invertible transformation, then $U_T$ preserves the norm on $L^1(\mu)$. 
\end{lemma}

\begin{proof}
Let $f\in L^1(\mu)$, then
\begin{align}
\|U_Tf\|_\mu&=\sup_{x\in X}|f(Tx)|N_\mu(x)\nn\\	
		&=\sup_{x\in X}|f(Tx)|N_\mu(Tx)=\|f\|_\mu.\nn 
\end{align}
\end{proof}

\begin{definition}
Let $T:(X,\cR_\mu,\mu)\to(X,\cR_\mu,\mu)$ and $S:(Y,\cR_\nu,\nu)\to(Y,\cR_\nu,\nu)$ be two invertible measure preserving transformations. They are called \emph{spectrally isomorphic} if there is an invertible linear isometry $W:L^1(\mu)\to L^1(\nu)$, such that
\[U_S\circ W=W\circ U_T.\]
\end{definition}

\begin{lemma}
Conjugacy implies spectral isomorphy.
\end{lemma}

\begin{proof}
Let $\phi:(\cR_\mu,\mu)\to(\cR_\nu,\nu)$ be the measure algebra isomorphism. Recall that $S(X)$ is the space of step function on $\cR_\mu$. Define a linear map \[U_\phi:S(X)\to S(Y), \sum_i^N\alpha_i\chi_{A_i}\mapsto\sum_i^N\alpha_i\chi_{\phi(A_i)}.\]
The map $U_\phi$ is invertible, since $\phi$ is a bijection. Let $f=\sum\alpha_i\chi_{A_i}$ be an element of $S(X)$ such that $\cap_iA_i=\emptyset$, then
\begin{align}
 \|U_\phi(f)\|_\nu&=\sup_{y\in\cup_i\phi(A_i)}|\alpha_i|\cdot\inf\{\|U\|_\nu:U\in\cR_\nu, y\in U\}\nn\\
		&=\sup_{x\in\cup_iA_i}|\alpha_i|\cdot\inf\{\|\phi(U)\|_\nu:U\in\cR_\nu, x\in U\}\nn\\
		&=\sup_{x\in\cup_iA_i}|\alpha_i|\cdot\inf\{\|U\|_\nu:U\in\cR_\mu, x\in U\}=\|f\|_\mu.\nn
\end{align}
So $U_\phi$ is an isometry. Extend $U_\phi$ by continuity. Because $S$ and $T$ are conjugate it follows that $U_S\circ U_\phi=U_\phi\circ U_T$.	 
\end{proof}

We do not know if spectral isomorphy induces conjugacy. However, the following lemma gives some conditions on a spectral isomorphism to come from a measure algebra isomorphism.

\begin{lemma}\label{conditions}
Let $W:L^1(\mu)\to L^1(\nu)$  be an isometry.
If for all bounded functions $f,g\in L ^1(\mu)$, $W(fg)=W(f)W(g)$, and if for all bounded functions $f\in L^1(\mu)$, 
\[\int_Xfd\mu=\int_Y W(f)d\nu,\]
then there exists a measure algebra isomorphism $\phi:(\cR_\mu,\mu)\to(\cR_\nu,\nu)$ such that $W=U_\phi$.
\end{lemma}

\begin{proof}
For any $B\in\cR_\mu$, $W(\chi^2_B)=W(\chi_B)W(\chi_B)=W(\chi_B)$, and therefore, $W(\chi_B)$ only takes values in $\{0,1\}$. Because $W(\chi_B)$ is an integrable function in $L^1(\nu)$, there is a $A\in\cR_\nu$ such that $W(\chi_B)=\chi_A$. It follows that $W$ sends indicator functions to indicator functions. Define $\phi(B)=A$, with $B$ and $A$ as above. 
Because, $\mu(B)=\int\chi_Bd\mu=\int\chi_Ad\nu=\nu(A)$, it is only left to prove that $\phi$ preserves unions and complements. Take $B,C\in\cR_\mu$, then
\[\chi_{B\cup C}=\chi_B+\chi_C-\chi_B\chi_C,\]
and hence,
\[\chi_{\phi(B\cup C)}=\chi_{\phi(B)}+\chi_{\phi(C)}-\chi_{\phi(B)}\chi_{\phi(C)}=\chi_{\phi(B)\cup\phi(C)}.\]
To prove that $\phi$ preserves complements we first prove that $\phi(X)=Y$. Suppose that $\phi(X)=A$, and that there is a nonempty  $A'\subset Y\backslash A$ in $\cR_\nu$. Then there is a nonempty $D\in\cR_\mu$ such that $W^{-1}(\chi_A')=\chi_D$. Because $\phi$ preserves unions, $\chi_A=W(\chi_D+\chi_{X\backslash D})=\chi_{A'}+\chi_{\phi(X\backslash D)}$, and in particular, $A'\subset A$, which is a contradiction. Hence, $\phi(X)=Y$, and therefore for any $B\in\cR_\mu$, 
\[\chi_{Y}=\chi_{B}+\chi_{\phi(X\backslash B)},\]
and thus $\phi(X\backslash B)=Y\backslash\phi(B)$.
\end{proof}

\section{Entropy}
One of the strongest invariants for dynamical systems is the entropy. We will discuss a version of measure-theoretic entropy for non-Archimedean measures. In some special cases, including some of our examples, this non-Archimedean entropy coincides with the topological entropy. Our treatment is based on that of Walters in \cite{Walters}. All logarithms in this section are in base 2. 

\subsection*{Partitions, subalgebras and entropy.}  
Let $(X,\mu,\cR_\mu)$ be a probability space.
\begin{definition}
A \emph{partition} of $(X,\mu,\cR_\mu)$ is a collection of disjoint elements of $\cR_\mu$ which cover $X$. 
\end{definition}

A partition $\alpha$ is called finite if it contains only finitely many elements. The set of partitions is a partial ordered space, where $\alpha\prec\beta$ means that each element of $\alpha$ is a union of elements of $\beta$. The collection which exists of all, possibly empty, unions of elements of $\alpha$ forms a finite subalgebra of $\cR_\mu$. This algebra is denoted by $\cA(\alpha)$. There is a one-to-one correspondence between finite subalgebras and partitions in the following way. Let $\cC=\{C_1,...,C_k\}$ be a finite subalgebra, then the nonempty intersections of the form $B_1\cap B_2\cap...\cap B_k$ where $B_i=C_i$ or $B_i=X\backslash C_i$ is a partition denoted by $\alpha(\cC)$. This correspondence respects the partial order in the sense that $\cA(\alpha)\subset\cA(\beta)$ if and only if $\alpha\prec\beta$.

\begin{definition}
For two partitions $\alpha=(A_1,...,A_k)$ and $\beta=(B_1,...,B_n)$ define
\[\alpha\vee\beta=\{A_i\cap B_j:1\leq i\leq k,1\leq j\leq n\},\]
which itself is a partition. 
\end{definition}

The operation $\vee$ is defined similarly for finite subalgebras such that $\alpha(\cA\vee\cC)=\alpha(\cA)\vee\alpha(\cC)$ and $\cA(\alpha\vee\beta)=\cA(\alpha)\vee\cA(\beta)$. For a finite partition $\alpha$ we define the significant part $\M(\alpha)=\{A\in\alpha:\|A\|>0\}$. Let $M(\alpha)$ be the cardinality of $\M(\alpha)$, this number is submultiplicative in the sense that $M(\alpha\vee\beta)\leq M(\alpha)M(\beta)$.

\begin{definition}
Let $\cA$ be a finite subalgebra, and let $\alpha(\cA)=(A_1,...,A_n)$ be the corresponding partition, then the \emph{measure entropy} with respect to $\cA$ is defined by 
\[H_\mu(\cA)=\min_{A\in\M(\cA)}\|A\|\log M(\alpha).\] 
\end{definition}

\begin{lemma}\label{eigenschappenentropy}
The measure entropy possesses the following properties:
\begin{enumerate} 
\item $H_\mu(\cA\vee\cB)\leq H_\mu(\cA)+H_\mu(\cB)$,
\item for any measure preserving transformation $T$, $H_\mu(\cA)=H_\mu(T\inv\cA)$. 
\end{enumerate}
\end{lemma}

\begin{proof}
\begin{enumerate}
 \item \begin{align}
        H_\mu(\cA\vee\cB)&=\min_{A_i\cap B_j\in\M(\alpha\vee\beta)}\|A_i\cap B_j\|\log M(\alpha\vee\beta)\nn\\
   &\leq\min_{A_i\cap B_j\in\M(\alpha\vee\beta)}(\|A_i\|,\|B_j\|)(\log M(\alpha)+\log M(\beta))\nn\\   &\leq H_\mu(\cA)+H_\mu(\cB)\nn
\end{align}
\item Both $\|\cdot\|$ and $M$ are invariant under $T$.
\end{enumerate}
\end{proof}
\begin{definition}
Let $T:X\to X$ be a measure preserving transformation, then the \emph{measure entropy} of $T$ with respect to a  finite subalgebra $\cA$ is
\[h_\mu(T,\cA)=\lim_{n\to\infty}\frac{1}{n}H_\mu(\cA\vee T\inv\cA\vee...\vee T^{-(n-1)}\cA)\] 
\end{definition}
To prove that  $h_\mu(T,\cA)$  exists we need the following lemma.
\begin{lemma}[\cite{Walters}, Theorem 4.4]\label{limietrijtje}
If $\{a_n\}_{n\in N}$ is a sequence in $\R$ which satisfies $a_n\geq 0$, $a_{n+m}\leq a_n+a_m$, then $\lim_{n\to\infty}\frac{a_n}{n}$ exists and is equal to $\inf_n\frac{a_n}{n}$. 
\end{lemma}

\begin{proposition}\label{limietbestaat}
The limit $h_\mu(T,\cA)=\lim_{n\to\infty}\frac{1}{n}H_\mu(\cA\vee T^{-1}\cA\vee...\vee T^{-(n-1)}\cA)$ exists.
\end{proposition}

\begin{proof}
Let $a_n=H_\mu(\cA\vee T^{-1}\cA\vee...\vee T^{-(n-1)}\cA)$, then by lemma (\ref{eigenschappenentropy}) 
\begin{align}
 a_{n+m}&=H_\mu(\cA\vee T^{-1}\cA\vee...\vee T^{-(n-1)}\cA\vee T^{-n}\cA\vee...T^{-(n+m-1)})\nn\\
&\leq a_n+H_\mu(T^{-n}\cA\vee...T^{-(n+m-1)})\nn\\
&=a_n+a_m.
\end{align}
Then the result follows by application of lemma (\ref{limietrijtje}).
\end{proof}

\begin{example}\label{positieve entropy}
We consider $\mu_{(-2,-2,5)}$ with values in $\Q_5$  on $\{0,1,2\}^\N$ like in example (\ref{shift}), and compute the entropy with respect to several partitions. First, let $\alpha=\{U_0,U_1,U_2\}$. Then 
\[H_\mu(\sigma,\cA(\alpha))=\lim_{n\to\infty}\frac{1}{n}5^{-n}\log(3^n)=0.\]
Second, let $\beta=\{A_0=U_0, A_1=U_1\cup U_2\}$. Since elements of $U_{i_0...i_n}$ with $i_0...i_n\in \{0,1\}^{n+1}$ are contained in $T^{-n}A_{i_n}\vee...\vee A_{i_0}$,
\[H_\mu(\sigma,\cA(\beta))=\lim_{n\to\infty}\frac{1}{n}\log(2^n)=\log(2).\]  
\end{example}

\begin{definition}
The \emph{measure-theoretic entropy} of a measure preserving transformation $T$ is 
\[h_\mu(T)=\sup_\cA h_\mu(T,\cA),\]
where the supremum is taken over all finite subalgebras.
\end{definition}

\begin{remark}
The Kolmogorov-Sinai Theorem (e.g.\cite{Walters}, 4.9) for classical measure entropy states that if $\cA$ is a finite algebra such that $\vee_{n=-\infty}^\infty T^n\cA\circeq\cB$, where $\cB$ is the $\sigma$-algebra, then $h_\mu(T)=h_\mu(T,\cA)$. Here $\cC\circeq\cB$ means that for any $C\in\cC$ there is a $B\in\cB$ such that $\mu(C\bigtriangleup B)=0$, and vice versa. Example \ref{positieve entropy} shows that such theorem is not true in this non-Archimedean setting. The partition $\alpha$ generates the covering ring, however, the entropy with respect to $\beta$ is greater then the entropy with respect to $\alpha$.    
\end{remark}

\begin{proposition}
The measure-theoretic entropy is invariant under conjugacy.
\end{proposition}

\begin{proof}
This follows because $\|\cdot\|$ is invariant under measure algebra isomorphism and moreover, for any finite subalgebra, $M(\alpha)$ is invariant.
\end{proof}

\subsection*{Connections with topological entropy}
An other form of entropy is topological entropy, introduced by Adler, Konheim and McAndrew \cite{AKM}. In this section we study the topological entropy on the zero dimensional topology induced by a separating covering ring $\cR_\mu$ and the connections between measure entropy and topological entropy. 

The analog of a measure preserving transformation in topological dynamics is simply a homeomorphism $T:X\to X$. Two such homeomorphisms $T:X\to X$, $S:Y\to Y$  are called \emph{topologically conjugate} if there is a homeomorphism $\phi:X\to Y$ such that $\phi\circ T=S\circ\phi$. 

Let $X$ be a compact space. Any open cover $\cU$ of $X$ has a finite subcover. Denote with $N(\cU)$ the least cardinality of all subcovers of $\cU$. 

\begin{definition}
The \emph{topological} entropy with respect to an open cover $\cU$ is
\[\Htop(\cU)=\log N(\cU).\]
\end{definition}
The collection of open covers of $X$ behaves in many senses similar to the collection of partitions. For instance, it is partially ordered.  For two open covers $\cU$ and $\cW$ we write that $\cU<\cW$ if every member of $\cW$ is a subset of a member of $\cU$, we say that $\cW$ is a \emph{refinement} of $\cU$. If $\cU<\cW$, then
\[\Htop(\cU)\leq\Htop(\cW).\]
The \emph{join} of two covers is defined by 
\[\cU\vee\cW=\{U\cap W: U\in\cU,W\in\cW\}.\]
In particular, $\cU<\cU\vee\cW$, and $N(\cU\vee\cW)\leq N(\cU)N(\cW)$ and it follows that 
\[\Htop(\cU\vee\cW)\leq\Htop(\cU)+\Htop(\cW).\] 
\begin{definition}
Let  $T:X\to X$ be a homeomorphism, then the \emph{topological entropy of $T$ with respect to $\cU$} is given by: 
\[\htop(T,\cU)=\lim_{n\to\infty}\frac{1}{n}\Htop(\cU\vee T\inv\cU\vee...\vee T^{-(n-1)}\cU).\]
\end{definition}
The proof that this limit exist is very similar to the proof of proposition (\ref{limietbestaat}) together with the observation that $\Htop(T\inv\cU)=H(	\cU)$.

\begin{definition}
The \emph{topological entropy} of a homeomorphism $T:X\to X$ is defined by:
\[\htop(T)=\sup_{\cU}\htop(T,\cU),\]
where the supremum is taken over all open covers of $X$.  
\end{definition}
Note that if $\cU'\subset\cU$ is a finite subcover, then $\cU<\cU'$ and $\htop(T,\cU)\leq\htop(T,\cU')$. Therefore, it is sufficient to take the supremum over all finite open covers of $X$. 

\subsection*{Partitions and coverings} Let us consider a probability space $(X,\mu,\cR_\mu)$. A partition $\alpha$ of $X$ is itself an open cover, because elements of $\alpha$ are required to be elements of $\cR_\mu$. The two partial orders coincide on collection of partitions, i.e., $\alpha\prec\beta$ if and only if $\alpha<\beta$. The other way round, given a finite open cover $\cU=\{U_1,...,U_k\}$ one can construct a partition $\alpha(\cU)$ by taking the nonempty sets of the form
\[A_1\cap...\cap A_k,\]
where  $A_i$ is $U_i$ or $X\backslash U_i$. 

\begin{lemma}
Let $\cU,\cW$ be two finite open covers, then
\begin{align}
 \alpha(\cU\vee\cW)&=\alpha(\cU)\vee\alpha(\cW),\nn\\
 \text{and if }\cU<\cW &\text{ then }\alpha(\cU)\prec\alpha(\cW).\nn
\end{align}
\end{lemma}

\begin{proof}
The fist assertion follows from the identity $(U\cap W)\backslash(A\cup B)=U\backslash A\cap W\backslash B$. Note for the second assertion that, since $\alpha(\cU)$ and $\alpha(\cW)$ are both partitions, it sufficient to prove that any $B\in\alpha(\cW)$  is contained in an $A\in\alpha(\cU)$. This is clear, for let $B=B_1\cap...\cap B_l$ be a nonempty element of $\alpha(\cW)$, where $B_i$ is $W_i$ or $X\backslash W_i$. Then $B$ is contained in at least one of the sets of the form $A_1\cap...\cap A_l$ where $B_i \subset A_i $.   
\end{proof}

\begin{theorem}\label{entropy}
Let $(X,\cR_\mu,\mu)$ be a compact probability space satisfying $\|X\|=1$, and let $T:X\to X$ be a measure preserving transformation, then 
\[h_\mu(T)\leq\htop(T),\]
and equality holds if $X_0=\emptyset$ and if for any nonempty set $A\in\cR_\mu$, $\|A\|=1$. 
\end{theorem}

\begin{proof}
Note that for any finite open cover $\cU$ with corresponding partition $\alpha(\cU)$
\begin{eqnarray*}
\min_{A\in\alpha(\cU)}\|A\|\log(M(\alpha(\cU)))&\leq&\log(N(\alpha(U))),\text{ and therefore, }\nn\\
h_\mu(T,\alpha(\cU))&\leq& \htop(T,\alpha(\cU)).\nn
\end{eqnarray*}

Therefore, $h_\mu(T)=\sup_{\cU}h_\mu(T,\alpha(\cU))\leq\sup_{\cU}\htop(T,\alpha(\cU))=\htop(T)$, where the last equality follows from the fact that all partitions are open covers. This proves the first part of the theorem. The second part follows, because if $X_0=\emptyset$ and if for all $A\in\cR_\mu$ $\|A\|=1$, then all inequalities above are equalities. 
\end{proof}

\begin{remarks}
The shift map in example \ref{shift} satisfies the conditions of this theorem if one take the probability vector $\bf{q}=(q_0,...,q_{p-1})$ such that all $|q_i|=1$. 
\end{remarks}

\bibliographystyle{amsplain}

\end{document}